\documentclass[11pt]{amsart}
\usepackage{graphicx}
\usepackage{amssymb,amsmath,amsgen,amsopn,amsxtra,amsbsy,amscd}
\usepackage{mathtools}
\usepackage{subcaption}
\numberwithin{equation}{section}

\newtheorem{proposition}{Proposition}[section]
\newtheorem{theorem}[proposition]{Theorem}

\theoremstyle{definition}

\newtheorem{remark}[proposition]{Remark}

\numberwithin{equation}{section}

\begin{document}
\title{Grain Boundary Grooving in a Bicrystal with Passivation Coating}

\author{H.~Kalantarova$^\ast$,
        L.~Klinger$^\ast$ and  E.~Rabkin$^\ast$}
\address{$^\ast$Department of Materials Science and Engineering, Technion-Israel Institute of Technology, 3200003 Haifa, Israel}


\date{\today}

\maketitle

\begin{abstract}
We use the sixth order linear parabolic equation 
\begin{equation}
\frac{\partial y}{\partial t}=B\left(\alpha\frac{\partial^{6}y}{\partial x^{6}}-\frac{\partial^{4}y}{\partial x^{4}}\right),\ x\in \mathbb{R}_{+},\ t>0,\nonumber
\end{equation}
proposed by Rabkin and describing the evolution of a solid surface covered with a thin, inert and fully elastic passivation layer, to analyze the grain boundary groove formation on initially flat surface. We derive the corresponding boundary conditions and construct an asymptotic representation of the solution to this initial boundary value problem when $\alpha$ is small, by applying the theory of singular perturbation. We illustrate the effect of passivation film near and far from a grain boundary groove.
\end{abstract}

\maketitle

\section{Introduction}\label{sec:intr}
Grain boundaries (GBs) are planar defects in polycrystalline solids that separate grains with different crystallographic orientations. Their thermodynamic properties are described by the GB energy, $\gamma_{gb}$, defined as the work required to reversibly increase the GB area by $1 m^{2}$. Wherever the GB emerges at the free surface of the polycrystal, the characteristic GB groove (sharp valley) is formed, provided the temperature is high enough for the atoms of the solid to migrate by diffusion. The thermodynamic driving force for the GB grooving is the reduction of the GB area and, as a consequence, of the total energy of all surfaces and GBs in the system. The thermal GB grooving has been widely employed in Physical Metallurgy for visualization of microstructure via thermal etching \cite{Shuttleworth1946}, \cite{Chinn2002}, \cite{Hackl2017}. Several mechanisms of the thermal GB grooving have been proposed, and it is generally accepted that for small lateral dimensions of the grooves (typically, smaller than $10 \mu m$) the surface diffusion dominates \cite{Mullins1959}. The theory of the GB grooving  controlled by surface self-diffusion has been developed by Mullins \cite{Mullins1957}, and since his seminal work it has been widely employed for determining the ratio of the GB and surface energies \cite{Amouyal2005}, and surface self-diffusion coefficients \cite{Bonzel1990}. 

It is generally believed that depositing a thin refractory passivation layer on the surface of metals prevents the GB grooving because such a layer does not allow the surface mass transport of metal atoms by surface diffusion. Thin refractory (ceramic) passivation layers are also naturally forming on the surface of many metals due to oxidation in ambient air (i.e. on Al, Ta, Nb, Cr and Ti). For example, it was reported that surface topography of thin Al films reflects their microstructure during film deposition in the oxygen-free environment. Once the film becomes exposed to oxygen and the thin surface layer of alumina is formed, the surface topography of the film becomes “frozen”. Subsequent heat treatments result in grain growth in the film, yet the surface topography does not change and does not correspond anymore to the actual microstructure of the film \cite{Nik2016}. However, recent studies have demonstrated that metal self-diffusion along the metal-ceramic interfaces may be much faster than bulk self-diffusion, and comparable with other types of short-circuit diffusion in solids \cite{Kumar2018}, \cite{Barda2019}. Also, recent solid state dewetting experiments with thin Al films deposited on sapphire substrate indicated that Al self-diffusion along the interface between Al and its native oxide may be fast enough to enable significant morphology evolution of the film \cite{Hieke2017}. An indirect evidence for intensive metal self-diffusion along the interface between the metal and thin passivation ceramic layer was obtained during the studies of growth of Au nanowhiskers from the thin Au films with alumina coating produced by atomic layer deposition \cite{Kosinova2018}, and during the studies of the interface roughness evolution in thin Cu films encapsulated in SiO$_2$ \cite{Warren2012}. Thus, these recent works indicate that under certain circumstances the GB grooving at the metal surface with a thin passivation layer is possible, since the metal atoms can diffuse along the metal-passivation (ceramic) interface. However, with an absence of any atom mobility on the side of refractory coating, any shape change on the side of the metal should inevitably result in elastic bending of the coating. Such bending modifies the chemical potential of mobile metal atoms at the metal-coating interface. For example, infinitely stiff refractory coating should suppress any surface morphology evolution even in the case when metal atoms can move along the metal-coating interface. Rabkin has considered the underlying physics and derived the sixth-order partial differential equation describing topography evolution of the passivated surface in the small-slope approximation \cite{Rabkin2012}. He then applied this equation to describe the flattening of sinusoidal surface perturbation by interface diffusion. In the present work, we will employ Rabkin’s equation 
\begin{equation}\label{a1}
\frac{\partial y}{\partial t}=B\left(\alpha\frac{\partial^{6}y}{\partial x^{6}}-\frac{\partial^{4}y}{\partial x^{4}}\right),\quad x\in \mathbb{R}_{+},\quad t>0,
\end{equation} 
to describe the process of GB grooving in a bicrystal with refractory surface passivation layer together with the initial condition 
\begin{equation}\label{a0}
y(x,0)\equiv 0,
\end{equation}
and the boundary conditions, which are derived in Section \ref{sec:model},
\begin{eqnarray}
&&\label{a2}
\frac{\partial y}{\partial x}(0,t)-\alpha\frac{\partial^{3}y}{\partial x^{3}}(0,t)=\frac{\gamma_{gb}}{2(\gamma_{i}+\gamma_{s})},\\
&&\label{a3}
\frac{\partial^{3}y}{\partial x^{3}}(0,t)-\alpha\frac{\partial^{5}y}{\partial x^{5}}(0,t)=0,\\
&&\label{a5}\frac{\partial^{2}y}{\partial x^{2}}(0,t)=0,\\
&&\label{a4}
\lim_{x\rightarrow\infty}\frac{\partial^{i}y}{\partial x^{i}}(x,t)=0,\quad i=1,2,3,
\end{eqnarray}
where $B$ is the Mullins' coefficient, $\alpha$ is a small parameter, $\gamma_{gb}$ is the GB energy, $\gamma_{i}$ is the interface energy and $\gamma_{s}$ is the surface stress of the coating.  The fact that the surface stress, rather than surface energy is a relevant parameter for the outer surface of the coating is related to the fact that we allow only elastic deformation of the coating and, therefore, its surface structure changes with the deformation. The situation with the interface is different, since atoms migration on the metal side enables some degree of relaxation upon the coating deformation. For the sake of simplicity we accepted the interface energy as a relevant parameter. Surface energy and stress of a solid are, generally, different from each other \cite{Frolov2009}. They coincide in the case of liquid surfaces. The Mullins' coefficient $B$ and the parameter $\alpha$ are defined as
\begin{equation}
B=\frac{D_{i}n\Omega^{2}(\gamma_{i}+\gamma_{s})}{kT},\quad \alpha=\frac{E h^{3}}{12(1-\nu^{2})(\gamma_{i}+\gamma_{s})},\nonumber
\end{equation} 
where $D_{i}$ is the interface diffusion coefficient of the metal, $n$ is the number of mobile atoms per $m^{2}$ of the interface, $\Omega$ is the atomic volume of the metal atoms, $k$ is the Boltzmann constant, $T$ is the temperature, $E$ is the Young's modulus of the coating, $h$ is its thickness and $\nu$ is the Poisson's ratio.  A schematic of a passivated bicrystal with a GB and a corresponding groove is shown in Fig.~\ref{fig:GB_groove}. 
\begin{figure}[h!]
  \centering
  \includegraphics[width=.57\textwidth]{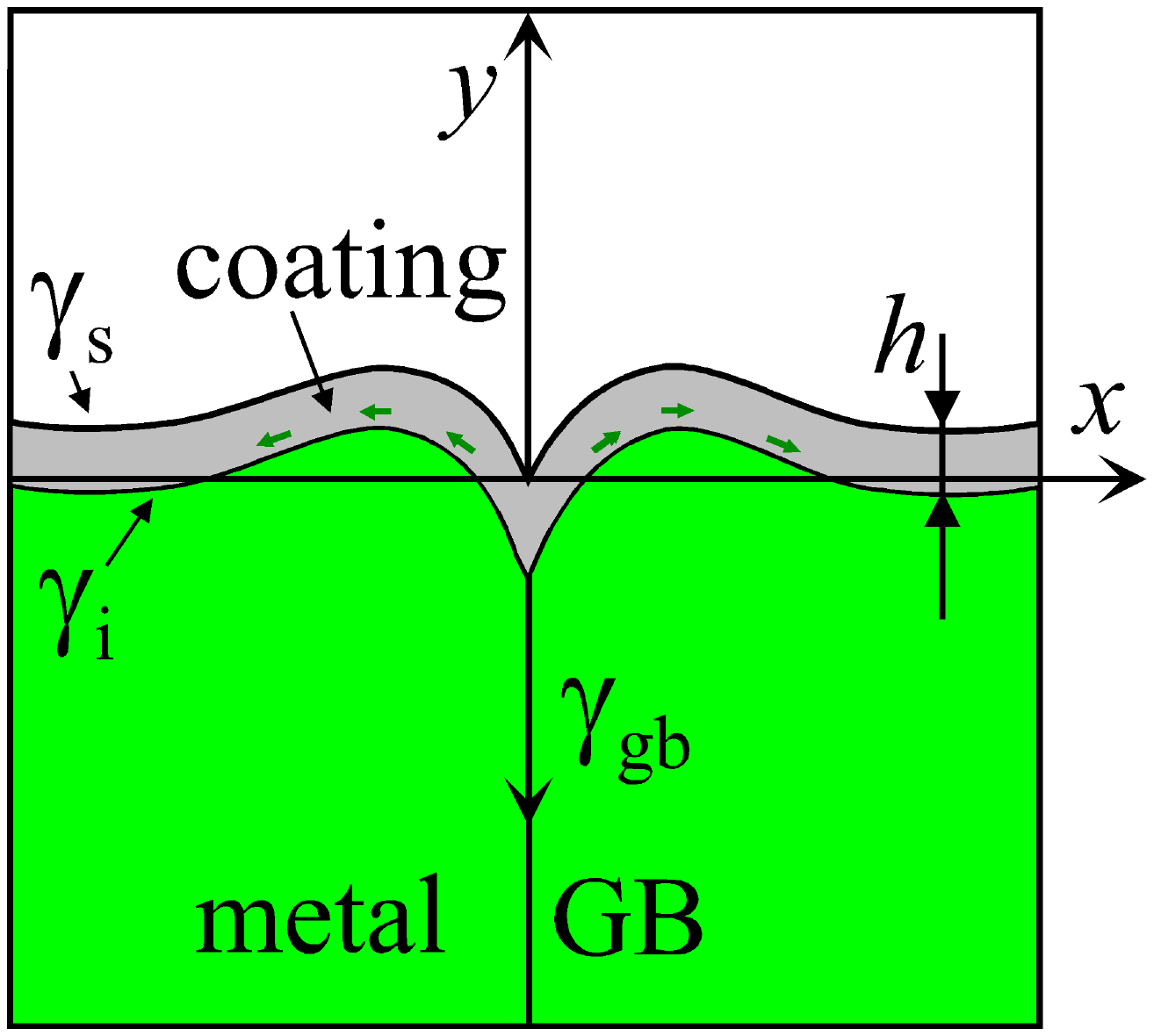}
 \caption{\footnotesize{Schematic of a metal bicrystal with passivating coating of thickness $h$. The diffusion flux of metal atoms along the metal-coating interface controlling the rate of groove growth is schematically shown by dark green arrows.}}
 \label{fig:GB_groove}
 \end{figure}
 
\begin{remark}
Note that as the slope of the surface at the groove root gets larger the error introduced by linearization, via small slope approximation, may be of the same order of magnitude as the error introduced by the effect of an elastic passivation layer, \cite{Robertson1971}, \cite{Tritscher1995}. Fortunately, for the most GB grooves observed in metals the aforementioned slope is less than $1/6$, and hence the linearized model can safely be employed.
\end{remark} 

\begin{remark}
We further note that $\alpha$ has dimension $m^{2}$. In the description of the model when we describe $\alpha$ as ''small'', we mean that $\alpha$ is small relative to $(Bt_{c})^{1/2}$ which is a square of the typical width of a groove after annealing for the time $t_{c}$.
\end{remark}

In order to study the behavior of the solutions of \eqref{a1} for small $\alpha$, we treat the initial boundary value problem \eqref{a1}-\eqref{a4} on the semi-infinite line $x\in \mathbb{R}_{+}$ as a singular perturbation of Mullins' linear surface diffusion equation (ME)
\begin{equation}\label{me}
\frac{\partial y}{\partial t}=-B\frac{\partial^{4}y}{\partial x^{4}},\quad x\in\mathbb{R}_{+},\quad t>0,
\end{equation}
together with the following initial and boundary conditions
\begin{eqnarray}
\label{me:ic}&&y(x,0)=0,\\
\label{me:bc}&&\frac{\partial y}{\partial x}(0,t)=\frac{m}{2},\quad \frac{\partial^{3}y}{\partial x^{3}}(0,t)=0,\quad \lim_{x\rightarrow\infty}\frac{\partial^{i}y}{\partial x^{i}}(0,t)=0,\quad t>0,
\end{eqnarray}
where $0<m<1/3$ (in the case of metals). \eqref{a1}-\eqref{a4} is called the \textit{singular problem} and \eqref{me}-\eqref{me:bc} is called the \textit{degenerate problem}.

A standard singular perturbation argument, \cite{Hinch}, \cite{Nayfeh}, consists of constructing a uniformly valid approximation of the solution to the singular problem which is valid throughout the domain, by adding two expansions, which are called outer expansion and inner expansion  and then by subtracting their common form $y_{overlap}$ which is valid in the overlap (intermediate) region. The outer expansion, $y_{out}$, is dominant in the outer region where the effect of the term with small parameter, namely $\alpha B\frac{\partial^{6}y}{\partial x^{6}}$, in \eqref{a1} is negligible, while the inner expansion, $y_{in}$,  is dominant in the boundary layer and the corner layer where the same term has a substantial effect  and is retained, see Fig.\ref{fig:layers}. 
 
In the outer region, the singular problem reduces to the degenerate problem, whose solution is given by

\begin{multline}
y_{out}(x,t;\alpha)=\frac{m}{2}x-\frac{mx^{2}}{4\sqrt{2}(Bt)^{1/4}\Gamma\left(\frac{3}{4}\right)}\prescript{}{1}{F}_{3}^{}\left(\frac{1}{4};\frac{3}{4},\frac{5}{4},\frac{3}{2};\frac{x^{4}}{256Bt}\right)\\
-\frac{m(Bt)^{1/4}}{2\sqrt{2}\Gamma\left(\frac{5}{4}\right)}\prescript{}{1}{F}_{3}^{}\left(-\frac{1}{4};\frac{1}{4},\frac{1}{2},\frac{3}{4};\frac{x^{4}}{256Bt}\right)+O(\alpha)\nonumber
\end{multline}
where the functions $\prescript{}{1}{F}_{3}^{}(a_{1};b_{1}, b_{2},b_{3};\cdot)$ with $a_{1}$, $b_{1}$, $b_{2}$, $b_{3}\in \mathbb{R}$ denote generalized hypergeometric functions, \cite{NIST}, \cite{Kalantarova2019}. Note that
\begin{equation}\label{ap:out}
y(x,t)\approx y_{out}(x,t;\alpha)
\end{equation}
is a good approximation except near $x=0$ since it does not satisfy \eqref{a5}. Thus, we anticipate a boundary layer near $x=0$, yielding an asymptotic expansion of the form
\begin{equation}
y(x,t)\approx y_{uniform}(x,t;\alpha)=y_{out}(x,t;\alpha)+y_{in}(x,t;\alpha)-y_{overlap}(x,t;\alpha),\nonumber
\end{equation}
where the uniform (composite) approximation $y_{uniform}$ satisfies \eqref{a2}-\eqref{a4}.

To determine the inner expansion, $y_{in}$, we transform the singular differential equation by variables that depend on $\alpha$ and that stretch the subregions of the boundary and the corner layer, see Fig \ref{fig:layers}.
\begin{figure}[ht]
  \centering
  \includegraphics[width=.47\textwidth]{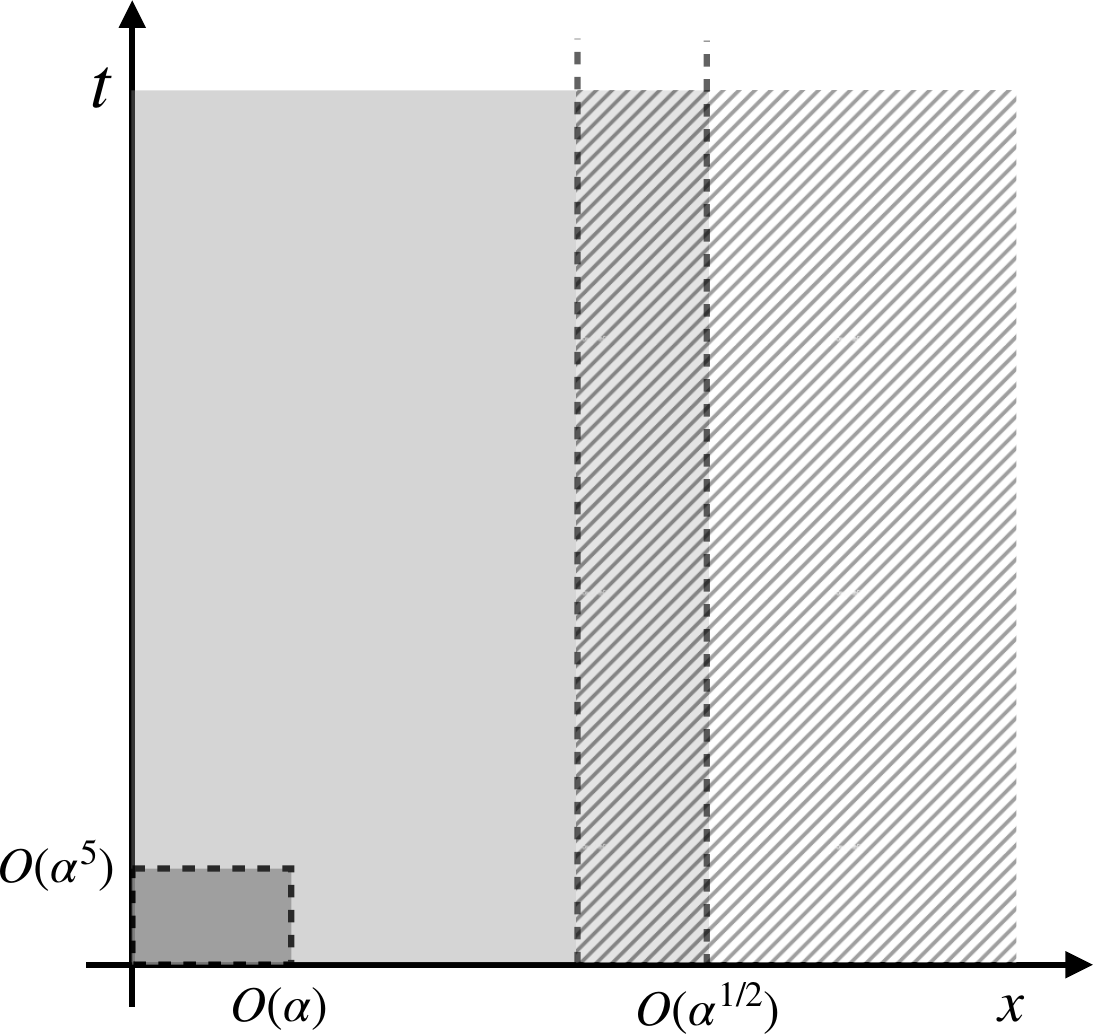}
 \caption{\footnotesize{Dark gray rectangular region represents the corner layer (small $t$, small $x$), semi-infinite gray region represents the boundary layer (small $x$) and the region filled with oblique lines is the outer region.}}
 \label{fig:layers}
 \end{figure}
The following equations are obtained by the least degeneracy principle
\begin{equation}\label{eq:b}
\frac{\partial^{6}y_{b}}{\partial \xi^{6}}-\frac{\partial^{4}y_{b}}{\partial \xi^{4}}=0,
\end{equation} 
\begin{equation}\label{eq:c}
\frac{\partial y_{c}}{\partial \tau}=B\frac{\partial^{6}y_{c}}{\partial \zeta^{6}},
\end{equation}
which are satisfied by the boundary layer and the corner layer correction terms, respectively. The derivation of the boundary and corner correction terms, together with a complete and detailed singular perturbation argument is presented in Section \ref{sec:perturbation}.

Finally, in section 4 we discuss plausible values for the physical parameters and provide groove profile plots which illustrate that the results derived in section 3 are physically meaningful. 

\section{The model}\label{sec:model}
In this section, we present the derivation of the differential equation that describes the evolution of a groove on a metal surface coated by a thin elastic film, which was originally derived in \cite{Rabkin2012} and establish the boundary conditions. To do so, we formulate the equilibrium configuration as a variational problem using the principles of the elasticity theory, \cite{Landau1970}.

The main difficulty in analyzing the GB grooving in a bicrystal with a thin elastic surface coating is associated with handling the elastic singularity at the root of the GB groove. We will bypass this difficulty by noting that refractory layers are usually brittle and tend to fracture at large elastic strains. Thus we will assume that the passivation layer is fractured at the root of the GB groove, and that the two halves of the bicrystal can be considered independently. Hence it is sufficient to consider the profile $y(x,t)$ for $x\in(0,\infty)$. The total energy for $x> 0$ is given by
\begin{multline}\label{e1}
G=\frac{1}{2}\gamma_{gb}y(0)+(\gamma_{i}+\gamma_{s})\int_{0}^{\infty}\sqrt{1+\left(\frac{\partial y}{\partial x}\right)^{2}}dx+\\
+\frac{Eh^{3}}{24(1-\nu^{2})}\int_{0}^{\infty}\left(\frac{\partial^{2}y}{\partial x^{2}}\right)^{2}dx
\end{multline}
where the first term on the right hand side corresponds to the grain boundary energy, the second term to the interface energy between the metal and the elastic film and the energy of the external film surface, while the third term describes the energy related with elastic deformation of the film. 

We consider the  first variation of  $G$  at $y$ in the direction of $\varphi\not=0$, which we denote by $\delta G$. The $\delta G$ is defined as
\begin{equation}\label{e2b}
\langle \delta G[y]; \varphi\rangle:=\left.\frac{d}{d\varepsilon}G[y+\varepsilon\varphi]\right|_{\varepsilon=0}
\end{equation}
where $\langle\cdot,\cdot\rangle$ is the standard $L^{2}$ product. Assuming that the functions involved are sufficiently smooth, we obtain

\begin{multline}\label{e2d}
\langle \delta G[y]; \varphi\rangle=\frac{1}{2}\gamma_{gb}\varphi(0)+(\gamma_{i}+\gamma_{s})\left.\frac{\frac{\partial y}{\partial x}\varphi}{\sqrt{1+\left(\frac{\partial y}{\partial x}\right)^{2}}}\right|_{0}^{\infty}\\
+(\gamma_{i}+\gamma_{s})\int_{0}^{\infty}\left(-\frac{\frac{\partial^{2} y}{\partial x^{2}}}{\sqrt{1+\left(\frac{\partial y}{\partial x}\right)^{2}}}-\frac{\left(\frac{\partial y}{\partial x}\right)^{2}\frac{\partial^{2} y}{\partial x^{2}}}{(1+\left(\frac{\partial y}{\partial x}\right)^{2})^{3/2}}\right)\varphi dx+\\
+(\gamma_{i}+\gamma_{s})\alpha\left(\left.\frac{\partial^{2} y}{\partial x^{2}}\frac{\partial\varphi}{\partial x}\right|_{0}^{\infty}-\left.\frac{\partial^{3} y}{\partial x^{3}}\varphi\right|_{0}^{\infty}+\int_{0}^{\infty} \frac{\partial^{4} y}{\partial x^{4}}\varphi dx\right),\nonumber
\end{multline}
where
\begin{equation}\label{e2e}
\alpha=\frac{Eh^{3}}{12(1-\nu^{2})(\gamma_{i}+\gamma_{s})}.
\end{equation}
The equilibrium state of the system is described by zero energy variation $\delta G=0$. Making small slope approximation, that is assuming $\frac{\partial y}{\partial x}\ll 1$, the equilibrium configuration is given by the differential equation
\begin{equation}
\frac{\partial^{2}y}{\partial x^{2}}-\alpha \frac{\partial^{4}y}{\partial x^{4}}=0,\nonumber
\end{equation}
supplemented by the following boundary conditions
\begin{eqnarray}
\label{e10}&&\frac{\partial y}{\partial x}(0,t)-\alpha\frac{\partial^{3}y}{\partial x^{3}}(0,t)=\frac{\gamma_{gb}}{2(\gamma_{i}+\gamma_{s})},\\
\label{e11}&&\frac{\partial^{2}y}{\partial x^{2}}(0,t)=0,\\
\label{e11a}&&\lim_{x\rightarrow\infty}\frac{\partial^{i}y}{\partial x^{i}}(x,t)=0,\quad i=0,\ 1,\ 2.
\end{eqnarray}

The chemical potential of the atoms at the interface, which we denote by $\mu$, can be obtained via
\begin{equation}\label{e2}
\mu=\Omega \delta G.
\end{equation} 
Substituting the resulting expression for $\delta G$ yields
\begin{equation}\label{e3}
\mu=\Omega (\gamma_{i}+\gamma_{s})\left(-\frac{\partial^{2}y}{\partial x^{2}}+\alpha \frac{\partial^{4}y}{\partial x^{4}}\right).
\end{equation}
The diffusion flux of atoms along the interface is proportional to the gradient of the chemical potential
\begin{equation}\label{e5}
j=-\frac{D_{i}n}{kT}\frac{\partial \mu}{\partial x}=\frac{D_{i}n\Omega(\gamma_{i}+\gamma_{s})}{kT}\left(\frac{\partial^{3}y}{\partial x^{3}}-\alpha \frac{\partial^{5}y}{\partial x^{5}}\right).
\end{equation}
Since the velocity of the displacement of the atoms along the interface is proportional to the divergence of the flux, we obtain
\begin{equation}\label{e6}
\frac{\partial y}{\partial t}=-\Omega\frac{\partial j}{\partial x}=-B\left(\frac{\partial^{4}y}{\partial x^{4}}-\alpha\frac{\partial^{6}y}{\partial x^{6}}\right),
\end{equation}
where							
\begin{equation}\label{e7}
B=\frac{D_{i}n\Omega^{2}(\gamma_{i}+\gamma_{s})}{kT}.
\end{equation}
We study the partial differential equation \eqref{e6} subject to the boundary conditions \eqref{e10}-\eqref{e11a}, which must hold for equilibrium state to be achieved, and the boundary condition given by
\begin{equation}\label{e8}
\frac{\partial^{3}y}{\partial x^{3}}(0,t)-\alpha\frac{\partial^{5}y}{\partial x^{5}}(0,t)=0,
\end{equation}
which follows from the zero flux assumption at the grain boundary.

\section{Perturbation Analysis}\label{sec:perturbation}

We next construct an asymptotic approximation to the solution of Rabkin's passivated surface evolution equation, \cite{Rabkin2012},
\begin{equation}\label{f0}
\frac{\partial y}{\partial t}=B\left(\alpha\frac{\partial^{6}y}{\partial x^{6}}-\frac{\partial^{4}y}{\partial x^{4}}\right),\quad x\in(0,\infty),\quad t>0,
\end{equation}
subject to the following boundary conditions (for derivation and physical meaning of these boundary conditions refer to Section \ref{sec:model})
\begin{eqnarray}
&&\label{f1a}
\frac{\partial y}{\partial x}(0,t)-\alpha\frac{\partial^{3}y}{\partial x^{3}}(0,t)=\frac{\gamma_{gb}}{2(\gamma_{i}+\gamma_{s})}\approx\frac{m}{2},\quad t>0,\\
&&\label{f1b}
\frac{\partial^{3}y}{\partial x^{3}}(0,t)-\alpha\frac{\partial^{5}y}{\partial x^{5}}(0,t)=0,\quad t>0,\\
&&\label{f1c}\frac{\partial^{2}y}{\partial x^{2}}(0,t)=0,\quad t>0,\\
&&\label{f1d}
\lim_{x\rightarrow\infty}\frac{\partial^{i}y}{\partial x^{i}}(x,t)=0,\quad i=0,\ 1,\ 2,\quad t>0.
\end{eqnarray}
and the initial planarity condition
\begin{equation}\label{si}
y(x,0)\equiv 0.
\end{equation}
Since \eqref{f0} does not remain a sixth order equation after setting the small parameter $\alpha$ to zero, we employ singular perturbation theory to study the above initial boundary value problem. Because of this reduction of the order, we cannot expect the solution of the resulting degenerate problem to satisfy all six boundary conditions \eqref{f1a}-\eqref{f1d}. We proceed by determining the outer expansion and the boundary condition that needs to be dropped, which lets us to determine where the boundary layers exist. Next, we construct the inner expansion. Finally, we construct a uniformly valid composite expansion.

\subsection*{Outer expansion:} First, we investigate the behavior of the solution of \eqref{f0} in the outer region, $y_{out}$, where the effect from the term containing the small parameter $\alpha$ is insignificant. We seek an outer expansion of second order (i.e. an outer expansion with an error term of order $O(\alpha^{3})$)
\begin{equation}\label{oeg}
y_{out}(x, t;\alpha)=y_{0}(x, t)+\alpha y_{1}(x, t)+\alpha^{2}y_{2}(x,t)+O(\alpha^{3}).
\end{equation}

To determine this expansion we substitute \eqref{oeg} into \eqref{f0} and then equate coefficients of $\alpha$ that have the same power, which yields
\begin{eqnarray}
\label{oe1}&&\frac{\partial y_{0}}{\partial t}=-B\frac{\partial^{4}y_{0}}{\partial x^{4}},\\
\label{oe3}&&\frac{\partial y_{r}}{\partial t}=-B\frac{\partial^{4}y_{r}}{\partial x^{4}}+B\frac{\partial^{6}y_{r-1}}{\partial x^{6}}\ \mbox{ for }\ r=1,2.
\end{eqnarray}
Note that \eqref{oe1} is simply ME given in \eqref{me}, \cite{Mullins1957}. Since ME is a fourth order parabolic equation, we can expect it to satisfy four of the six boundary conditions given in \eqref{f1a}-\eqref{f1d}. Because of the physical phenomenon that we are studying, the condition of asymptotic flatness must be satisfied. In \cite{Kalantarova2019}, Kalantarova et al. show that all the asymptotically decaying self-similar solutions of ME, subject to the initial planarity condition, that have the form
\begin{equation} \label{f3}
y(x,t) = (Bt)^{1/4} Z(u),\quad u=x/(Bt)^{1/4},
\end{equation}
where $Z(u)$ satisfies
\begin{equation}\label{f3a} 
Z^{(4)}(u)-\frac{1}{4}uZ'(u)+\frac{1}{4}Z(u)=0,
\end{equation}
are  generated by the following two parameter family of solutions
\begin{equation}\label{y0}
y_{0}(x,t)=c_{1}f_{1}(x,t)+c_{2}f_{2}(x,t),
\end{equation}
where $c_{1}$, $c_{2}$ are arbitrary constants and 
\begin{multline}\label{f1}
f_{1}(x,t)=\frac{x}{\sqrt{2}}-\frac{x^{2}}{2(Bt)^{1/4}\Gamma\left(\frac{3}{4}\right)}\prescript{}{1}{F}_{3}^{}\left(\frac{1}{4};\frac{3}{4},\frac{5}{4},\frac{3}{2};\frac{x^{4}}{256Bt}\right)\\
+\frac{x^{3}}{6\sqrt{2}\Gamma\left(\frac{1}{2}\right)(Bt)^{1/2}}\prescript{}{1}{F}_{3}^{}\left(\frac{1}{2};\frac{5}{4},\frac{3}{2},\frac{7}{4};\frac{x^{4}}{256Bt}\right),
\end{multline}
\begin{multline}\label{f2}
f_{2}(x,t)=\frac{(Bt)^{1/4}}{\Gamma\left(\frac{5}{4}\right)}\prescript{}{1}{F}_{3}^{}\left(-\frac{1}{4};\frac{1}{4},\frac{1}{2},\frac{3}{4};\frac{x^{4}}{256Bt}\right)-\frac{x}{\sqrt{2}}\\
+\frac{x^{3}}{6\sqrt{2}\Gamma\left(\frac{1}{2}\right)(Bt)^{1/2}}\prescript{}{1}{F}_{3}^{}\left(\frac{1}{2};\frac{5}{4},\frac{3}{2},\frac{7}{4};\frac{x^{4}}{256Bt}\right).
\end{multline}
Here  $\prescript{}{p}{F}_{q}^{}(a_{1},\ldots,a_{p};b_{1},\ldots,b_{q};u)$ denote the generalized hypergeometric functions  or the generalized hypergeometric series defined as
\begin{multline}\label{def:hf}
\prescript{}{p}{F}_{q}^{}(a_{1},\ldots,a_{p};b_{1},\ldots,b_{q};\nu)=\sum_{k=0}^{\infty}\frac{(a_{1})_{k}\ldots(a_{p})_{k}}{(b_{1})_{k}\ldots(b_{q})_{k}}\frac{\nu^{k}}{k!}\\
=1+\frac{a_{1}\ldots a_{p}}{b_{1}\ldots b_{q}}\nu+\frac{a_{1}(a_{1}+1)\ldots a_{p}(a_{p}+1)}{b_{1}(b_{1}+1)\ldots b_{q}(b_{q}+1)2!}\nu^{2}+\ldots,
\end{multline}
in which $\{a_{i}\}_{i=1}^{p}$, $\{b_{i}\}_{i=1}^{q}\in\mathbb{R}$ and
\begin{equation}
(\lambda)_{k}=\frac{\Gamma(\lambda+k)}{\Gamma(\lambda)}=\lambda(\lambda+1)\ldots(\lambda+k-1)\nonumber
\end{equation}
is the Pochhammer symbol, \cite{Askey2010}. One of the advantages of this representation over the series solution with recursively defined coefficients is that it can be plotted over arbitrarily large domains, which is useful for purpose of asymptotical matching. 

Substituting $y_{out}$ (defined in \eqref{oeg} where $y_{0}$ is given by \eqref{y0}-\eqref{f2}) into \eqref{f1a}-\eqref{f1c}, namely the boundary conditions at $x=0$, we get
\begin{equation}
\frac{c_{1}-c_{2}}{\sqrt{2}}=\frac{m}{2},\quad \frac{c_{1}+c_{2}}{\sqrt{2\pi}\sqrt{Bt}}=0,\quad -\frac{c_{1}}{(Bt)^{1/4}\Gamma\left(\frac{3}{4}\right)}=0.
\end{equation}
Observe that there exist no $c_{1}$, $c_{2}\in\mathbb{R}$ for which all of the above conditions are satisfied, thus there is a boundary layer near $x=0$. Before calculating $y_{1}$ and $y_{2}$ in \eqref{oeg}, we proceed with determining where the boundary layers lie, which are the regions where the term with the small parameter $\alpha$ is of the same order or larger than the other terms involved in the equation \eqref{f0}. Then, we calculate the boundary layer correction term, which will simplify the calculation of $y_{1}$ and $y_{2}$.

\subsection*{The boundary and corner layer:}  We transform \eqref{f0} by 
\begin{equation}
\tilde{y}(\xi,\tau)=y\left(\frac{x}{\alpha^{\mu}},\frac{t}{\alpha^{\nu}}\right),
\end{equation}
where $\mu, \nu>0$ are to be determined later. Substituting $\tilde{y}$ into \eqref{f0}, yields that
\begin{equation}\label{eq:yin}
\alpha^{-\nu}\frac{\partial \tilde{y}}{\partial \tau}=B\left(\alpha^{1-6\mu}\frac{\partial^{6}\tilde{y}}{\partial \xi^{6}}-\alpha^{-4\mu}\frac{\partial^{4}\tilde{y}}{\partial \xi^{4}}\right).
\end{equation}
We need to set the coefficient of the highest derivative to 1. This can be achieved for two different sets of values of $\mu$ and $\nu$:
\begin{equation}
\mu_{1}=\frac{1}{2},\quad \nu_{1}=0,
\end{equation}
and

\begin{equation}
\mu_{2}=1,\quad \nu_{2}=5.
\end{equation}
By setting $\mu=\mu_{1}=1/2$ and $\nu=\nu_{1}=0$ in \eqref{eq:yin}, we obtain the following boundary layer equation
\begin{equation}\label{eq:yinb}
\frac{\partial^{6}y_{b}}{\partial \xi^{6}}-\frac{\partial^{4}y_{b}}{\partial \xi^{4}}=\frac{\alpha^{2}}{B}\frac{\partial y_{b}}{\partial t},
\end{equation}
where
\begin{equation}\label{bl:var}
\xi=\frac{x}{\alpha^{1/2}},
\end{equation}
is the boundary layer variable. On the other hand, setting $\mu=\mu_{2}=1$ and $\nu=\nu_{2}=5$ leads to the following corner layer equation
\begin{equation}\label{eq:yinc}
\frac{\partial y_{c}}{\partial \tau}=B\frac{\partial^{6}y_{c}}{\partial \zeta^{6}}+\alpha B\frac{\partial^{4}y_{c}}{\partial\zeta^{4}},
\end{equation}
where
\begin{equation}\label{cl:var}
\tau=\frac{t}{\alpha^{5}}\quad\mbox{and}\quad\zeta=\frac{x}{\alpha},
\end{equation}
are corner layer variables. See Figure \ref{fig:layers} for illustration of regions of boundary and corner layers.

\bigskip

Here we use the technique of Bromberg, Vishik and Lusternik (see \cite[page 146]{Nayfeh}) to determine the composite expansion. The advantage of this technique is that it lets us to construct a uniformly valid expansion without the need for matching the outer and inner expansions. We assume that
\begin{eqnarray}
y(x,t;\alpha)&=&F(x,t;\alpha)+G(\xi,t;\alpha)+H(\zeta,\tau;\alpha)\nonumber\\
&=&\sum_{r=0}^{\infty}\alpha^{r}F_{r}(x,t)+\sum_{r=0}^{\infty}\alpha^{r/2}G_{r}(\xi,t)+\sum_{r=0}^{\infty}\alpha^{r}H_{r}(\zeta,\tau)\nonumber\\
&=&F_{0}(x,t)+G_{0}(\xi,t)+H_{0}(\zeta,\tau)+\alpha^{1/2}G_{1}(\xi,t)\nonumber\\
&&+\alpha[F_{1}(x,t)+G_{2}(\xi,t)+H_{1}(\zeta,\tau)]+\alpha^{3/2}G_{3}(\xi,t)\nonumber\\
\label{y:comp}&&+\alpha^{2}[F_{2}(x,t)+G_{4}(\xi,t)+H_{2}(\zeta,\tau)]+O(\alpha^{3})
\end{eqnarray}
where $\xi$ is the boundary layer variable defined in \eqref{bl:var} and $\zeta$, $\tau$ are the corner layer variables defined in \eqref{cl:var} and where
\begin{eqnarray}
\label{bl:dec}&&\lim_{\xi\rightarrow\infty}G(\xi,t;\alpha)=0,\\
\label{cl:dec}&&\lim_{\zeta\rightarrow\infty}H(\zeta,\tau;\alpha)=0,\\
\label{cl:dect}&&\lim_{\tau\rightarrow\infty}H(\zeta,\tau;\alpha)=0,
\end{eqnarray}
i.e., $G(\xi,t;\alpha)$ is negligible outside the boundary layer and $H(\zeta,\tau;\alpha)$ is negligible outside the corner layer.

It follows from \eqref{bl:dec} and \eqref{cl:dec} that
\begin{equation}
y_{out}(x,t;\alpha)=F(x,t;\alpha)=F_{0}(x,t)+\alpha F_{1}(x,t)+\alpha^{2}F_{2}(x,t)+O(\alpha^{3}).\nonumber
\end{equation}
Then, we have $F_{r}(x,t)\equiv y_{r}(x,t)$ for $r=0,1,2,...$, where $y_{r}(x,t)$ are defined in \eqref{oeg}.

On the other hand, 
\begin{eqnarray}
y_{in}(x,t;\alpha)&=&G(\xi,t;\alpha)+H(\zeta,\tau;\alpha)+y_{overlap}(x,t;\alpha)\nonumber\\
&=&y_{b}(x,t;\alpha)+y_{c}(x,t;\alpha),\nonumber
\end{eqnarray}
where $y_{b}$ is the boundary layer correction term that satisfies the equation \eqref{eq:yinb} and $y_{c}$ is the corner layer correction term that satisfies the equation \eqref{eq:yinc}.

\begin{remark}\label{cl:small}
For $t=O(\alpha^{5})$ and $x\in(0,\infty)$, we have 
\begin{equation}
y_{0}(x,t)=O(\alpha).\nonumber
\end{equation} 
That is, inside the corner layer the values that decaying solutions of ME attain are negligible. Then so is the contribution to the composite expansion from the corner layer correction term even inside the corner layer. Nonetheless, we will solve the corner layer equation for the sake of completeness of the analysis of the problem.
\end{remark}

\subsection*{Boundary layer correction term:}  Recall that $x=\alpha^{1/2}\xi$. Since the effect of the corner layer term on composite expansion is negligible, we have
\begin{multline}\label{y:bexp}
y_{b}(\xi,t;\alpha)=F_{0}(0,t)+G_{0}(\xi,t)+\alpha^{1/2}\left[\frac{\partial F_{0}(0,t)}{\partial x}\xi+G_{1}(\xi,t)\right]\\
+\alpha\left[\frac{1}{2}\frac{\partial^{2}F_{0}(0,t)}{\partial x^{2}}\xi^{2}+F_{1}(0,t)+G_{2}(\xi,t)\right]\\
+\alpha^{3/2}\left[\frac{1}{3!}\frac{\partial^{3}F_{0}(0,t)}{\partial x^{3}} \xi^{3}+\frac{\partial F_{1}(0,t)}{\partial x}\xi+G_{3}(\xi,t)\right]\\
+\alpha^{2}\left[\frac{1}{4!}\frac{\partial^{4}F_{0}(0,t)}{\partial x^{4}}\xi^{4}+\frac{1}{2!}\frac{\partial^{2}F_{1}(0,t)}{\partial x^{2}}\xi^{2}+F_{2}(0,t)+G_{4}(\xi,t)\right]\\
+O(\alpha^{3}).
\end{multline}
Substituting \eqref{y:bexp} into \eqref{eq:yinb} and then equating the coefficients of $\alpha^{r}$, for $r=0,\frac{1}{2}, 1, \frac{3}{2}, 2,$ we obtain
\begin{eqnarray}
\label{bl:eq1}&&\frac{\partial^{6}G_{r}}{\partial \xi^{6}}-\frac{\partial^{4}G_{r}}{\partial\xi^{4}}=0,\ \mbox{ for }\ r=0,1,2,3,\\
\label{bl:eq3}&&\frac{\partial^{6}G_{4}}{\partial \xi^{6}}-\frac{\partial^{4}G_{4}}{\partial\xi^{4}}-\frac{\partial^{4}F_{0}(0,t)}{\partial x^{4}}=\frac{1}{B}\left[\frac{\partial F_{0}(0,t)}{\partial t}+\frac{\partial G_{0}(\xi,t)}{\partial t}\right].
\end{eqnarray}
Solving \eqref{bl:eq1}-\eqref{bl:eq3} under the assumption \eqref{bl:dec}, we get
\begin{equation}\label{sol:b2}
G(\xi,t;\alpha)\approx\sum_{r=0}^{3}\alpha^{r/2}\beta_{r}(t)e^{-\xi}+\alpha^{2}\beta_{4}(t)G_{4}(\xi,t),
\end{equation}
where $\beta_{r}(t)$, $r=0,1,2,3,4$, and $G_{4}(\xi,t)$ are to be determined later.
 
\begin{remark}\label{bc:0}
Note that the boundary conditions \eqref{f1a}-\eqref{f1c} must be satisfied for all $t>0$. Then they must be satisfied by $y_{b}$ given in \eqref{y:bexp}, due to Remark \ref{cl:small} and \eqref{cl:dect}.
\end{remark}

Thus, we substitute \eqref{y:bexp} into \eqref{f1b}, which yields
\begin{equation}\label{f1o}
c_{1}=-c_{2}.
\end{equation}
Next, substituting \eqref{y:bexp} into \eqref{f1a} and using \eqref{f1o}, we arrive at
\begin{equation}\label{f2o}
c_{1}=-c_{2}=\frac{m}{2\sqrt{2}}.
\end{equation}
\begin{remark}
The conditions \eqref{f1o} and \eqref{f2o} are equivalent to the boundary conditions proposed by Mullins, \cite{Mullins1957},
\begin{equation}
\frac{\partial y_{0}}{\partial x}(0,t)=\frac{m}{2},\quad \frac{\partial^{3}y_{0}}{\partial x^{3}}(0,t)=0,\nonumber
\end{equation}
respectively. As a result, under these conditions $y_{0}$ is Mullins' solution
\begin{multline}\label{y:out}
y_{0}(x,t)=\frac{m}{2}x-\frac{mx^{2}}{4\sqrt{2}(Bt)^{1/4}\Gamma\left(\frac{3}{4}\right)}\prescript{}{1}{F}_{3}^{}\left(\frac{1}{4};\frac{3}{4},\frac{5}{4},\frac{3}{2};\frac{x^{4}}{256Bt}\right)\\
-\frac{m(Bt)^{1/4}}{2\sqrt{2}\Gamma\left(\frac{5}{4}\right)}\prescript{}{1}{F}_{3}^{}\left(-\frac{1}{4};\frac{1}{4},\frac{1}{2},\frac{3}{4};\frac{x^{4}}{256Bt}\right).\nonumber
\end{multline}
\end{remark}
\noindent Now that we have determined $y_{0}$, we proceed with the calculation of $y_{1}$ and $y_{2}$.

\subsection*{Outer expansion (continued):} It follows from substituting \eqref{y:bexp} into \eqref{f1a}-\eqref{si} that $y_{1}(x,t)$ must satisfy the following initial and boundary conditions
\begin{equation}
y_{1}(x,0)=0,\quad \frac{\partial y_{1}}{\partial x}(0,t)=0,\quad \frac{\partial^{3}y_{1}}{\partial x^{3}}(0,t)=0.
\end{equation}

We take the Fourier cosine transform of \eqref{oe3} for $r=1$ in $x$ and obtain
\begin{equation}\label{f7}
\frac{\partial Y_{1c}}{\partial t}(k,t)+Bk^{4}Y_{1c}(k,t)=-B\frac{m}{2}k^{4}e^{-Bk^{4}t}
\end{equation}
with

\begin{equation}\label{f7a}
Y_{1c}(k,0)=0,
\end{equation}
where
\begin{equation}\label{f8}
\mathcal{F}_{x}^{(c)}\{y(x,t)\}(k,t)\equiv Y_{c}(k,t)=\int_{0}^{\infty}y(x,t)\cos(kx)dx,
\end{equation}
represents the Fourier cosine transform of $y(x,t)$ with respect to the space variable $x$. The right hand side of \eqref{f7} follows from
\begin{equation}
\mathcal{F}_{x}^{(c)}\{y_{0}(x,t)\}(k,t)=\frac{m(e^{-k^{4}Bt}-1)}{2k^{2}},\nonumber
\end{equation}
which is calculated in \cite[(10) on page 127]{Martin2009} and the following property of Fourier cosine transform

\begin{multline}
\mathcal{F}_{x}^{(c)}\left\{\frac{\partial^{6}y}{\partial x^{6}}(x,t)\right\}(k,t)=\\
-k^{6}\mathcal{F}_{x}^{(c)}\{y(x,t)\}(k,t)-k^{4}\frac{\partial y}{\partial x}(0,t)+k^{2}\frac{\partial^{3}y}{\partial x^{3}}(0,t)-\frac{\partial^{5}y}{\partial x^{5}}(0,t).\nonumber
\end{multline}
Solving \eqref{f7}-\eqref{f7a} as a first order initial value problem in time we get
\begin{equation}\label{f9}
 Y_{1c}(k,t)=-Bt\frac{m}{2}k^{4}e^{-Bk^{4}t}.
\end{equation}
Taking the inverse Fourier cosine transform gives
\begin{eqnarray}
y_{1}(x,t)&=&\frac{2}{\pi}\int_{0}^{\infty}-Bt\frac{m}{2}k^{4}e^{-Bk^{4}t}\cos(kx)dk\nonumber\\
&=&-\frac{m \Gamma \left(\frac{1}{4}\right) }{16 \pi (B t)^{1/4}}\, _1F_3\left(\frac{5}{4};\frac{1}{4},\frac{1}{2},\frac{3}{4};\frac{x^4}{256 B t}\right)\nonumber\\
\label{yout:y1}&&-\frac{3 m x^2 \Gamma \left(-\frac{1}{4}\right) }{128  \pi (B t)^{3/4}}\, _1F_3\left(\frac{7}{4};\frac{3}{4},\label{f10}\frac{5}{4},\frac{3}{2};\frac{x^4}{256 B t}\right),
\end{eqnarray}
which is an asymptotically decaying function. Next, we repeat the same steps that we made in the derivation of $y_{1}$, \eqref{f10}, to derive $y_{r}$ for $r=2,3,\ldots$. In order to do so, we need to seek an outer solution of the form
\begin{equation}\label{q1}
y_{out}(x, t;\alpha)=\sum_{r=0}^{N}\alpha^{r}y_{r}(x,t)+O(\alpha^{N+1}).
\end{equation}
where $N>2$, instead of \eqref{oeg}. Then, additionally to \eqref{oe1} and \eqref{oe3} we must solve
\begin{equation}
\label{q2}\frac{\partial y_{r}}{\partial t}=-B\frac{\partial^{4}y_{r}}{\partial x^{4}}+B\frac{\partial^{6}y_{r-1}}{\partial x^{6}},\mbox{ for }r=2,3,\ldots, N,
\end{equation}
which is simply the same equation in \eqref{oe3} for higher indices and which results from substituting \eqref{q1} into \eqref{f0}. Moreover, $y_{r}(x,t)$ for $r=3,\ldots,N$ satisfy the following initial and boundary conditions

\begin{equation}
\label{q4}y_{r}(x,0)=0,\quad \frac{\partial y_{r}}{\partial x}(0,t)=0,\quad \frac{\partial^{3}y_{r}}{\partial x^{3}}(0,t)=0,\mbox{ for }r=3,\ldots,N,
\end{equation}
which follow from substituting \eqref{q1} into (3.2)-(3.6). We solve \eqref{q2} again by using Fourier cosine transforms. First, we take the Fourier cosine transform of \eqref{q2} with respect to $x$
\begin{equation}\label{q3}
\frac{\partial Y_{rc}}{\partial t}(k,t)+Bk^{4}Y_{rc}(k,t)=(-B)^{r}t^{r-1}\frac{m}{2\times(r-1)!}k^{6r-2}e^{-Bk^{4}t}
\end{equation}
then solve the resulting ordinary differential equation
\begin{equation}\label{q5}
Y_{rc}(k,t)=(-Bt)^{r}\frac{m}{2\times r!}k^{6r-2}e^{-Bk^{4}t}
\end{equation}
and finally recover the solution of \eqref{q2} by taking the inverse Fourier cosine transform of \eqref{q5}
\begin{multline}\label{q6}
y_{rc}(x,t)=(-1)^{r}\left[\frac{m \Gamma \left(\frac{3 r}{2}-\frac{1}{4}\right) \, _1F_3\left(\frac{3 r}{2}-\frac{1}{4};\frac{1}{4},\frac{1}{2},\frac{3}{4};\frac{x^4}{256 B t}\right)}{4 \pi (B t)^{\frac{r}{2}-\frac{1}{4}} r!}\right.\\
\left.-\frac{m x^2 \Gamma \left(\frac{3 r}{2}+\frac{1}{4}\right) \, _1F_3\left(\frac{3 r}{2}+\frac{1}{4};\frac{3}{4},\frac{5}{4},\frac{3}{2};\frac{x^4}{256 B t}\right)}{8 \pi (B t)^{\frac{r}{2}+\frac{1}{4}} r!}\right].
\end{multline}
Note that, inserting $r=1$ in \eqref{q6} yields \eqref{yout:y1}. Then, we have
\begin{multline}\label{q8}
y_{out}(x,t;\alpha)=\frac{m}{2}x-\frac{mx^{2}}{4\sqrt{2}(Bt)^{1/4}\Gamma\left(\frac{3}{4}\right)}\prescript{}{1}{F}_{3}^{}(\frac{1}{4};\frac{3}{4},\frac{5}{4},\frac{3}{2};\frac{x^{4}}{256Bt})\\
-\frac{m(Bt)^{1/4}}{2\sqrt{2}\Gamma\left(\frac{5}{4}\right)}\prescript{}{1}{F}_{3}^{}(-\frac{1}{4};\frac{1}{4},\frac{1}{2},\frac{3}{4};\frac{x^{4}}{256Bt})\\
+\sum_{r=1}^{N}(-\alpha)^{r}\left[\frac{m \Gamma \left(\frac{3 r}{2}-\frac{1}{4}\right) \, _1F_3\left(\frac{3 r}{2}-\frac{1}{4};\frac{1}{4},\frac{1}{2},\frac{3}{4};\frac{x^4}{256 B t}\right)}{4 \pi (B t)^{\frac{r}{2}-\frac{1}{4}} r!}\right.\\
\left.-\frac{m x^2 \Gamma \left(\frac{3 r}{2}+\frac{1}{4}\right) \, _1F_3\left(\frac{3 r}{2}+\frac{1}{4};\frac{3}{4},\frac{5}{4},\frac{3}{2};\frac{x^4}{256 B t}\right)}{8 \pi (B t)^{\frac{r}{2}+\frac{1}{4}} r!}\right]+O(\alpha^{N+1}).
\end{multline}
Recalling Remark \ref{bc:0}, it follows from the requirement of the fulfillment of the boundary condition \eqref{f1c} that
\begin{equation}
\beta_{0}(t)=0,\quad \beta_{1}(t)=0,\quad \beta_{2}(t)=\frac{m}{2\sqrt{2}(Bt)^{1/4}\Gamma\left(\frac{3}{4}\right)}\quad \mbox{ and }\quad \beta_{3}(t)=0.\nonumber
\end{equation}
Now, we can solve \eqref{bl:eq3}
\begin{equation}
G_{4}(\xi,t)=\beta_{4}(t)e^{-\xi},\nonumber
\end{equation}
where
\begin{equation}
\beta_{4}(t)=-\frac{m \Gamma \left(\frac{7}{4}\right)}{4 \pi  (B t)^{3/4}},\nonumber
\end{equation}
due to \eqref{f1c}. Thus
\begin{equation}\label{q5}
G(\xi,t;\alpha)=\alpha\frac{ m}{2\sqrt{2}(Bt)^{1/4}\Gamma(\frac{3}{4})}e^{-\xi}-\alpha^{2}\frac{m \Gamma \left(\frac{7}{4}\right)}{4 \pi  (B t)^{3/4}}e^{-\xi}+O(\alpha^{3}),
\end{equation}
where $\xi=\frac{x}{\sqrt{\alpha}}$ is the boundary layer variable.

\subsection*{Corner layer correction term:} The corner layer equation \eqref{eq:yinc}
 \begin{equation}
\frac{\partial y_{c}}{\partial \tau}=B\frac{\partial^{6}y_{c}}{\partial \zeta^{6}},\nonumber
\end{equation}
has self-similar solutions of the form
\begin{equation}\label{yc:ss}
y_{c}(\zeta,\tau)=(B\tau)^{r}V(\zeta/(B\tau)^{1/6}),
\end{equation}
where $V=V(w)$ satisfies
\begin{equation}\label{fyc}
V^{(6)}(w)+\frac{1}{6}wV'(w)-rV(w)=0,\quad w=\zeta/(B\tau)^{1/6}.
\end{equation}
Using the theory of generalized hypergeometric equations, \cite{Askey2010}, we show that the linear ordinary differential equation above has the following fundamental set of solutions

\begin{eqnarray}
\label{fv1}&&v_{1}(w)\!=\, _1F_5\left(-r;\frac{1}{6},\frac{1}{3},\frac{1}{2},\frac{2}{3},\frac{5}{6};-\frac{w^6}{6^{6}}\right),\\
\label{fv2}&&v_{2}(w)=w\, _1F_5\left(\frac{1}{6}-r;\frac{1}{3},\frac{1}{2},\frac{2}{3},\frac{5}{6},\frac{7}{6};-\frac{w^6}{6^{6}}\right),\\
\label{fv3}&&v_{3}(w)\!=w^2 \, _1F_5\left(\frac{1}{3}-r;\frac{1}{2},\frac{2}{3},\frac{5}{6},\frac{7}{6},\frac{4}{3};-\frac{w^{6}}{6^{6}}\right),\\
\label{fv4}&&v_{4}(w)=w^3 \, _1F_5\left(\frac{1}{2}-r;\frac{2}{3},\frac{5}{6},\frac{7}{6},\frac{4}{3},\frac{3}{2};-\frac{w^6}{6^{6}}\right),
\end{eqnarray}
\begin{eqnarray}
\label{fv5}&&v_{5}(w)\!=w^4 \, _1F_5\left(\frac{2}{3}-r;\frac{5}{6},\frac{7}{6},\frac{4}{3},\frac{3}{2},\frac{5}{3};-\frac{w^6}{6^{6}}\right),\\
\label{fv6}&&v_{6}(w)=w^5 \, _1F_5\left(\frac{5}{6}-r;\frac{7}{6},\frac{4}{3},\frac{3}{2},\frac{5}{3},\frac{11}{6};-\frac{w^6}{6^{6}}\right).
\end{eqnarray}
Using Laplace transform methods to solve \eqref{eq:yinc} and taking \eqref{yc:ss} and \eqref{fv1}-\eqref{fv6} into consideration, we find that

\begin{equation}
\begin{bmatrix}
y_{c1}(\zeta,\tau)\\[0.3em]
y_{c2}(\zeta,\tau)\\[0.3em]
y_{c3}(\zeta,\tau)\\[0.3em]
y_{c4}(\zeta,\tau)\\[0.3em]
y_{c5}(\zeta,\tau)\\[0.3em]
y_{c6}(\zeta,\tau)
\end{bmatrix}
=(B\tau)^{r}\begin{bmatrix}
 1   &  1                        &  1                         & 1  & 1                           &  1 \\[0.5em]
 1   & \frac{1}{2}          &-\frac{1}{2}           & -1  &-\frac{1}{2}           &\frac{1}{2}  \\[0.5em]
 0   &\frac{\sqrt{3}}{2} &  \frac{\sqrt{3}}{2}&  0  &-\frac{\sqrt{3}}{2}  &-\frac{\sqrt{3}}{2}   \\[0.5em]
 1   & -1                        &  1                         &-1  & 1                           & -1 \\[0.5em]
 1   &-\frac{1}{2}          &-\frac{1}{2}           &  1  &-\frac{1}{2}           &-\frac{1}{2}  \\[0.5em]
 0   &\frac{\sqrt{3}}{2} &-\frac{\sqrt{3}}{2} &  0  &\frac{\sqrt{3}}{2}   &-\frac{\sqrt{3}}{2}   
\end{bmatrix}
\begin{bmatrix}
 \frac{1}{0!\Gamma(1+r)}v_{1}(w)\\[0.3em]
 \frac{1}{1!\Gamma\left(\frac{5}{6}+r\right)} v_{2}(w)\\[0.7em]
 \frac{1}{2!\Gamma\left(\frac{4}{6}+r\right)}v_{3}(w)\\[0.7em]
 \frac{1}{3!\Gamma\left(\frac{1}{2}+r\right)}v_{4}(w)\\[0.7em]
 \frac{1}{4!\Gamma\left(\frac{1}{3}+r\right)}v_{5}(w)\\[0.7em]
 \frac{1}{5!\Gamma\left(\frac{1}{6}+r\right)}v_{6}(w)
\end{bmatrix}.\nonumber
\end{equation}
Furthermore, for $\tau>0$
\begin{equation}
\lim_{\zeta\rightarrow\infty}\lvert y_{c1}(\zeta,\tau)\rvert=\lim_{\zeta\rightarrow\infty}\lvert y_{c2}(\zeta,\tau)\rvert=\lim_{\zeta\rightarrow\infty}\lvert y_{c3}(\zeta,\tau)\rvert=\infty\nonumber
\end{equation}
and
\begin{equation}
\lim_{\zeta\rightarrow\infty} y_{c4}(\zeta,\tau)=\lim_{\zeta\rightarrow\infty} y_{c5}(\zeta,\tau)=\lim_{\zeta\rightarrow\infty} y_{c6}(\zeta,\tau)=0.\nonumber
\end{equation}
 
 \begin{theorem}\label{thm:cle}
 The sixth order parabolic equation \eqref{eq:yinc}
 \begin{equation}
\frac{\partial y_{c}}{\partial \tau}=B\frac{\partial^{6}y_{c}}{\partial \zeta^{6}},\nonumber
\end{equation}
subject to boundary conditions

\begin{eqnarray}
\label{bc:c1}&&\alpha\frac{\partial y_{c}}{\partial\zeta}(0,\tau)-\frac{\partial^{3}y_{c}}{\partial\zeta^{3}}(0,\tau)=0,\\
\label{bc:c2}&&\alpha\frac{\partial^{3}y_{c}}{\partial\zeta^{3}}(0,\tau)-\frac{\partial^{5}y_{c}}{\partial\zeta^{5}}(0,\tau)=0,\\
\label{bc:c3}&&\lim_{\zeta\rightarrow\infty}\frac{\partial^{i}y_{c}}{\partial\zeta^{i}}(\zeta,\tau)=0,\quad\tau>0,\quad i=0,1,2,
\end{eqnarray}
has the following nontrivial similarity solutions of the form
\begin{multline}
y_{c}(\zeta,\tau)=-\frac{\gamma}{3}\!\left[\alpha^{2}(B\tau)^{\frac{2}{3}}\Gamma\left(\!\frac{1}{6}+r\!\right)+\alpha(B\tau)^{\frac{1}{3}}\Gamma\left(\!\frac{1}{2}+r\!\right)+\Gamma\left(\!\frac{5}{6}+r\!\right)\right]y_{c4}(\zeta,\tau)\\
-\frac{\gamma}{3}\!\left[\alpha^{2}(B\tau)^{\frac{2}{3}}\Gamma\left(\!\frac{1}{6}+r\!\right)-2\alpha(B\tau)^{\frac{1}{3}}\Gamma\left(\!\frac{1}{2}+r\!\right)+\Gamma\left(\!\frac{5}{6}+r\!\right)\right]y_{c5}(\zeta,\tau)\\
-\frac{\gamma}{\sqrt{3}}\left[\alpha^{2}(B\tau)^{\frac{2}{3}}\Gamma\left(\!\frac{1}{6}+r\!\right)-\Gamma\left(\!\frac{5}{6}+r\!\right)\right]y_{c6}(\zeta,\tau).\nonumber
\end{multline}
where $r<-\frac{2}{3}$ and $\gamma=O(\alpha)$ is an arbitrary constant.
\end{theorem}

Observe that
\begin{multline}
\frac{\partial^{2}y_{c}}{\partial x^{2}}(0,t)=(\sqrt{3}-1)\gamma\frac{(Bt)^{r+\frac{1}{3}}\Gamma(r+\frac{1}{6})}{6\alpha^{5r+\frac{5}{3}}\Gamma(r+\frac{2}{3})}\\
-\frac{2}{3}\gamma\frac{(Bt)^{r}\Gamma(r+\frac{1}{2})}{\alpha^{5r+1}\Gamma(r+\frac{2}{3})}-(1+\sqrt{3})\gamma\frac{(Bt)^{r-\frac{1}{3}}\Gamma(r+\frac{5}{6})}{6\alpha^{5r+\frac{1}{3}}\Gamma(r+\frac{2}{3})}
\end{multline}
is negligible outside the corner layer.

Moreover, for each fixed time $t$ the integral of the composite expansion $y(x,t;\alpha)$ with respect to $x$ over $(0,\infty)$ is of negligible magnitude as expected due to the conservation of matter.
 
\section{Discussion}\label{sec:disc}

 In order to estimate the effect of the elastic coating on the GB groove profile for the experiment-relevant situation, we estimate parameter $\alpha$ for the thin $(h=5 nm)$ layer of alumina (native oxide) on the surface of metallic Al. Since there is a high uncertainty in elastic properties of amorphous alumina, we will employ the elastic constants of $\gamma-Al_{2}O_{3}$ (a product of crystallization of amorphous alumina once the thickness of the oxide layer exceeds some critical value): $E=253$ GPa and $\nu=0.24$, \cite{Gallas1994}. We will approximate the surface stress of $\gamma-$alumina by its surface energy, $\gamma_{s}=1.67 J/m^{2}$ \cite{Mchale1997}. Finally, the value of Al-alumina interface energy will be approximated by the value of the energy of the respective solid (sapphire)-liquid (Al) interface: $\gamma_{i}=1.2 J/m^{2}$ \cite{Levi2003}. With these values, $\alpha\approx 9.7\times 10^{-16} m^{2}$. 
\begin{remark}
According to the numerical study by Robertson \cite{Robertson1971} of nonlinear Mullins' equation the value of the normalized grain boundary groove profile at the origin (plotted in Fig.5 (b) as $-y_{0}/m(Bt)^{1/4}$) is $0.78$ for Mullins' linearized problem (1.7)-(1.9) and $0.77$ for Mullins' nonlinear problem. Thus, the error introduced by linearization is $1.3\%$ for the value of $m$ we consider. On the other hand, the effect of the elastic layer for the longest annealing time is $12.5\%$, which is much larger than the effect of linearization. Hence, ignoring the relatively small error due to linearization does not affect the validity of our asymptotic analysis.
\end{remark}
\begin{figure}[h!]
  \includegraphics[width=.8\textwidth]{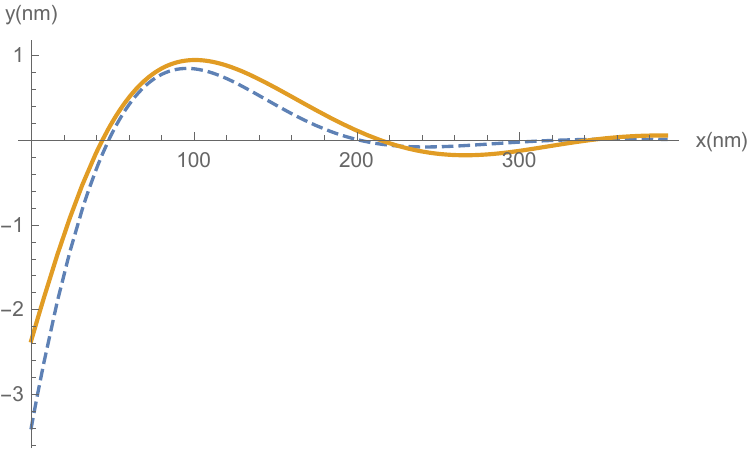}
 \caption{\footnotesize{The perturbation solution (thick line) is plotted together with Mullins' solution (dashed) for  $m=0.209$, $Bt=3\times10^{-30}$$m^{4}$ and $\alpha=9.7\times 10^{-16}$$m^{2}$}}
 \label{fig:t03}
 \end{figure}
\begin{figure}[h!]
 \includegraphics[width=.8\textwidth]{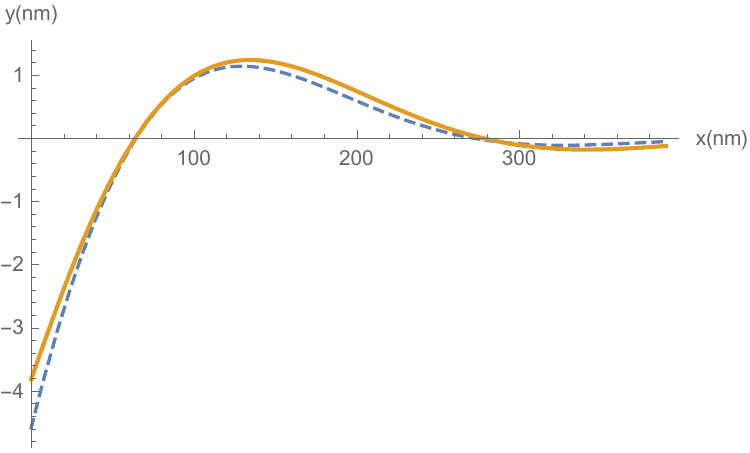}
 \caption{\footnotesize{The perturbation solution (thick line) is plotted together with Mullins' solution (dashed) for  $m=0.209$, $Bt=10^{-29}$$m^{4}$ and $\alpha=9.7\times 10^{-16}$$m^{2}$}}
 \label{fig:t05}
\end{figure}
  
\begin{figure}[h]
  \includegraphics[width=.9\textwidth]{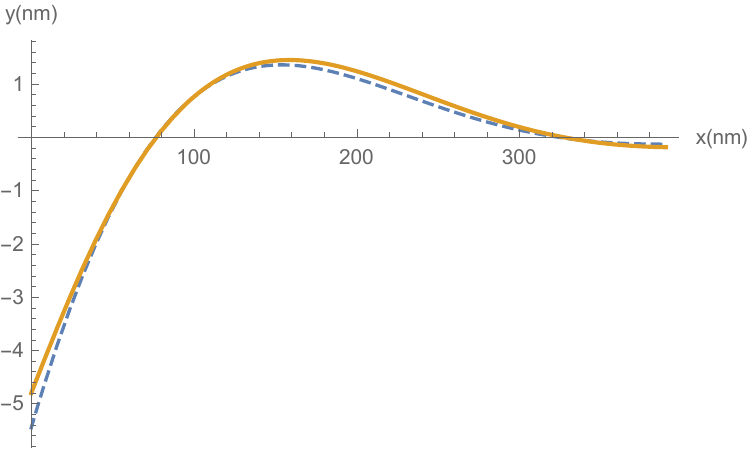}
 \caption{\footnotesize{The perturbation solution (thick line) is plotted together with Mullins' solution (dashed) for  $m=0.209$, $Bt=2\times10^{-29}$$m^{4}$ and $\alpha=9.7\times 10^{-16}$$m^{2}$}}
 \label{fig:t2}
\end{figure}

\begin{figure}[h!]
    \centering
    \begin{subfigure}[b]{0.45\textwidth}
        \includegraphics[width=\textwidth]{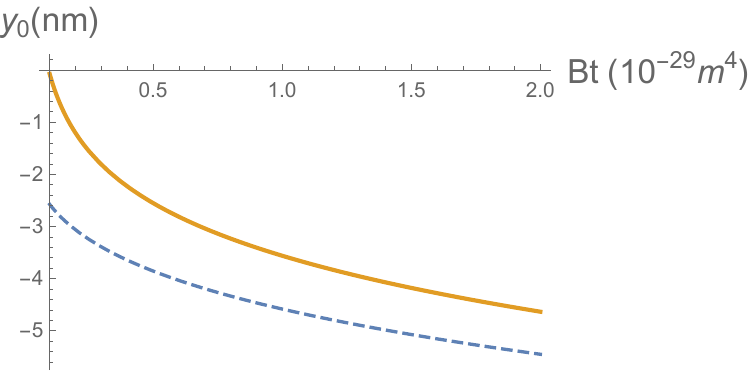}
        \caption{$\alpha=10.5\times10^{-16}$$m^{2}$}
        \label{fig:depth1}
    \end{subfigure} 
    \begin{subfigure}[b]{0.45\textwidth}
        \includegraphics[width=\textwidth]{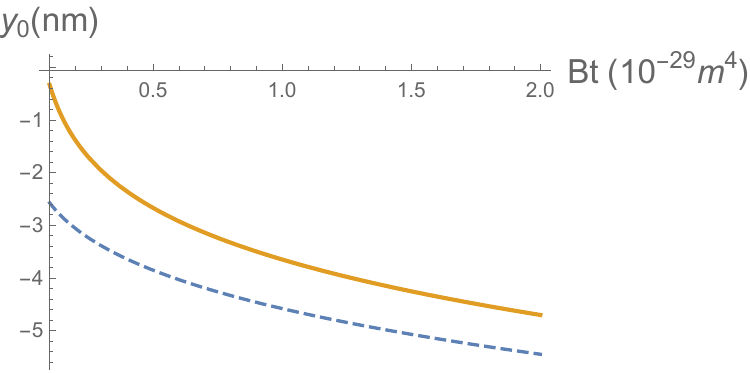}
        \caption{$\alpha=9.7\times10^{-16}$$m^{2}$}
        \label{fig:depth2}
    \end{subfigure}
       \begin{subfigure}[b]{0.45\textwidth}
        \includegraphics[width=\textwidth]{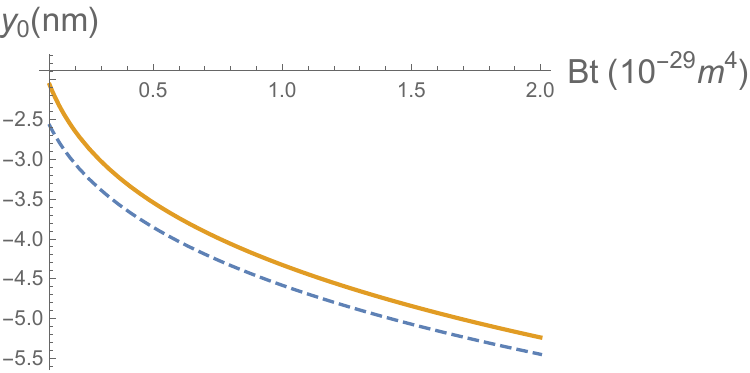}
        \caption{$\alpha=3\times10^{-16}$$m^{2}$}
        \label{fig:depth3}
    \end{subfigure}
       \begin{subfigure}[b]{0.45\textwidth}
        \includegraphics[width=\textwidth]{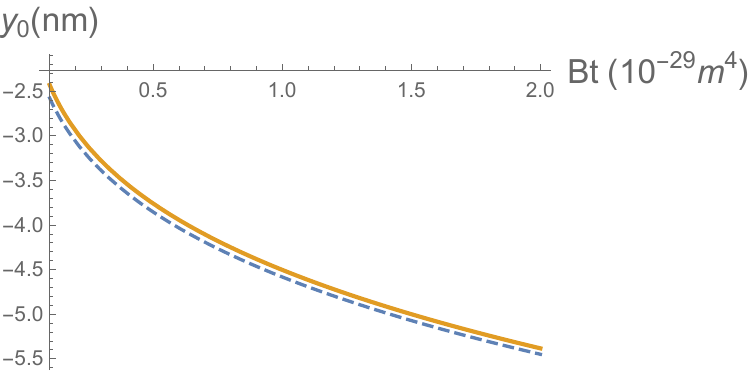}
        \caption{$\alpha=9.7\times10^{-17}$$m^{2}$}
        \label{fig:depth4}
    \end{subfigure}
    \caption{\footnotesize{The time evolution of groove depth $y_{0}(t):=y(0,t)$ of the perturbation solution (thick line) and Mullins' solution(dashed) are plotted for $m=0.209$} and different values of $\alpha$}\label{fig:d1d2d3}
\end{figure}
 
 Figs. \ref{fig:t03}-\ref{fig:t2} illustrate that the composite expansion $y(x,t;\alpha)$ defined in \eqref{y:comp} is a good approximation to the solution of our problem. One can see from these plots that the effect of elastic coating on the groove profile decreases with increasing annealing time and widening of the groove. Indeed, for the groove of 158.6 nm in half-width (see Fig.~\ref{fig:t2}) the composite expansion is basically indistinguishable from the original Mullins solution. For shorter times/narrower grooves, the effect of elastic coating is highest at the groove root, $x=0$ (see Figs. \ref{fig:t03}-\ref{fig:t2}), which is given by
\begin{multline}
y(0,t;\alpha)-y_{0}(0,t)=\sum_{r=1}^{2}(-\alpha)^{r}\frac{m \Gamma \left(\frac{3 r}{2}-\frac{1}{4}\right) }{4 \pi (B t)^{\frac{r}{2}-\frac{1}{4}} r!}\\
+\alpha\frac{ m}{2\sqrt{2}(Bt)^{1/4}\Gamma(\frac{3}{4})}-\alpha^{2}\frac{m \Gamma \left(\frac{7}{4}\right)}{4 \pi  (B t)^{3/4}}+O(\alpha^{3}).
\nonumber
\end{multline} 
This is understandable, since the elastic coating tries to “flatten” the surface profile. At the same time, the amplitudes of the primary maximum and a secondary minimum of the profile increase.  

Generally, for the estimated value of parameter $\alpha$ for thin passivating alumina film on the surface of metallic Al the difference between the original Mullins solution for unpassivated surface, and the solution derived in the present work with the aid of singular perturbation theory is rather small. This justifies in the hindsight the use of perturbation theory in handling of the problem in the cases of practical relevance. It does not mean, however, that the actual effect of passivating layer is always small. As mentioned above, very thick and stiff passivation layer can completely stop the GB grooving process. However, the perturbation theory cannot be employed for handling of such thick and stiff layers. Also, while our theory indicates that the difference of the shapes of the GB groove for unpassivated and passivated surfaces is small, it does not mean that the effect of passivation layer is small in absolute terms. Indeed, the parameter $B$ in equation \eqref{a1} scales with the metal-self-diffusion coefficient along the metal-coating interface, $D_{i}$, whereas in the classical Mullins model for unpassivated surfaces this parameter scales with the surface self-diffusion coefficient of the metal, $D_{s}$. Though several recent works indicate that self-diffusion of metal along the metal-oxide interface is much faster than bulk self-diffusion \cite{Barda2019}, \cite{Hieke2017}, \cite{Kosinova2018}, \cite{Kumar2018}, it is still much slower than self-diffusion of metal on unpassivated surface. At homologically low temperatures the difference between $D_{i}$ and $D_{s}$ can reach several orders of magnitude \cite{Barda2020}. Therefore, while the original Mullins model is a good approximation for describing the shape of the passivated GB grooves, the linear dimensions of such grooves may be significantly smaller than those of the grooves formed at unpassivated surface of the same material and the same annealing time.

Another interesting observation is that the coordinate $x_{m}$ of the primary groove profile maximum is only insignificantly affected by the presence of the coating. Since the value $x_{m}$ is employed for determining the surface self-diffusion coefficient from the profiles of GB grooves, one can conclude that it can be also employed for determining the self-diffusion coefficient of metal, $D_{i}$, along the metal-coating interface. 

\section{Conclusion}
In summary, we considered the problem of GB grooving in the initially planar metal bicrystal with inert elastic passivation coating. The GB-driven evolution of surface profile causes elastic bending of the coating, which in turn affects the driving force for interface diffusion determining the kinetics of the groove growth. We formulated the sixth-order linear partial differential equation describing the evolution of surface profile, and employed the variational method to determine the appropriate boundary conditions. Analysis of the problem in the framework of singular perturbation theory yielded an asymptotic solution converging to Mullins’ solution for unpassivated surfaces for long annealing times. For short annealing times the depth of the groove formed in passivated bicrystal is shallower than in its unpassivated counterpart. Also, we demonstrated that the original Mullins’ analysis can be employed for determining the self-diffusion coefficient of metal along the metal-coating interface.  \vskip6pt

\appendix
\section{Derivation of the Corner Layer Correction Term}
\begin{proof}[Proof of Theorem \ref{thm:cle}]
Guided by the fact that the equation \eqref{eq:yinc} has the following scaling symmetry, namely, given any solution $y_{c}(\zeta,\tau)$ to \eqref{eq:yinc},
\begin{equation}
y_{c\lambda}=\lambda^{r}y_{c}(\lambda\zeta,\lambda^{6}\tau)\quad\mbox{for any}\ \lambda>0,\ r\in\mathbb{R},
\end{equation}
is also a solution of \eqref{eq:yinc}, we seek $y_{c}(\zeta,\tau)$ in the form
\begin{equation}\label{yc:selfsim}
y_{c}(\zeta,\tau)=(B\tau)^{r}V(\zeta/(B\tau)^{1/6}),
\end{equation}
where $V(w)$ satisfies
\begin{equation}\label{eq:hyp}
V^{(6)}(w)+\frac{1}{6}wV'(w)-rV(w)=0,\quad w=\zeta/(B\tau)^{1/6},
\end{equation}
and
\begin{eqnarray}
\label{hyp:bc1}&&\alpha^{2}(B\tau)^{2/3}V'(0)=\alpha(B\tau)^{1/3}V'''(0)=V^{(5)}(0),\\
\label{hyp:bc2}&&\lim_{\zeta\rightarrow\infty}(B\tau)^{r}V(\zeta/(B\tau)^{1/6})=0,\quad \tau>0,
\end{eqnarray}
as in \cite{Mullins1957}. The equation \eqref{eq:hyp} is obtained by substituting \eqref{yc:selfsim} into \eqref{eq:yinc} and the boundary conditions \eqref{hyp:bc1}-\eqref{hyp:bc2} are obtained by substituting \eqref{yc:selfsim} into \eqref{bc:c1}-\eqref{bc:c3}.

Introducing a new variable $u$ such that 
\begin{equation}
w(u)=\frac{1}{6}u^{1/6},\nonumber
\end{equation}
in \eqref{eq:hyp} transforms it into a generalized hypergeometric differential equation [GHDE], which has a well-established theory, \cite{NIST}. Once transformed into GHDE solving the resulting equation is technical but straightforward. The calculations are essentially identical with that of corresponding result in \cite[Appendix A]{Kalantarova2019} and therefore will be omitted. The fundamental set of solutions of \eqref{eq:hyp} is given by
\begin{eqnarray}
\label{v1}&&v_{1}(w)\!=\, _1F_5\left(-r;\frac{1}{6},\frac{1}{3},\frac{1}{2},\frac{2}{3},\frac{5}{6};-\frac{w^6}{6^{6}}\right),\\
\label{v2}&&v_{2}(w)=w\, _1F_5\left(\frac{1}{6}-r;\frac{1}{3},\frac{1}{2},\frac{2}{3},\frac{5}{6},\frac{7}{6};-\frac{w^6}{6^{6}}\right),\\
\label{v3}&&v_{3}(w)\!=w^2 \, _1F_5\left(\frac{1}{3}-r;\frac{1}{2},\frac{2}{3},\frac{5}{6},\frac{7}{6},\frac{4}{3};-\frac{w^{6}}{6^{6}}\right),
\end{eqnarray}

\begin{eqnarray}
\label{v4}&&v_{4}(w)=w^3 \, _1F_5\left(\frac{1}{2}-r;\frac{2}{3},\frac{5}{6},\frac{7}{6},\frac{4}{3},\frac{3}{2};-\frac{w^6}{6^{6}}\right),\\
\label{v5}&&v_{5}(w)\!=w^4 \, _1F_5\left(\frac{2}{3}-r;\frac{5}{6},\frac{7}{6},\frac{4}{3},\frac{3}{2},\frac{5}{3};-\frac{w^6}{6^{6}}\right),\\
\label{v6}&&v_{6}(w)=w^5 \, _1F_5\left(\frac{5}{6}-r;\frac{7}{6},\frac{4}{3},\frac{3}{2},\frac{5}{3},\frac{11}{6};-\frac{w^6}{6^{6}}\right),
\end{eqnarray}
and hence its general solution is
\begin{equation}\label{sol:hyp}
V(w)=\sum_{i=1}^{6}C_{i}v_{i}(w),\quad w=\frac{\zeta}{(B\tau)^{1/6}},\quad C_{i}\in\mathbb{R}.
\end{equation}

Thus it follows from \eqref{sol:hyp} and \eqref{yc:selfsim} that, if $y_{c}(\zeta,\tau)$ is a self-similar solution to \eqref{eq:yinc}, then it can be expressed as
\begin{equation}\label{sol:gss}
y_{c}(\zeta,\tau)=(B\tau)^{r}\sum_{i=1}^{6}C_{i}v_{i}(w),
\end{equation}
where $\omega=\zeta/(B\tau)^{1/6}$, and where $C_{i}$, $i=1,\ldots, 6$ are arbitrary constants and $v_{i}(w)$ are linearly independent entire functions. Moreover,
\begin{equation}
\lim_{\zeta\rightarrow\infty}\lvert v_{i}(\zeta/(B\tau)^{1/6})\rvert=\infty\quad\mbox{for}\quad\tau>0.\nonumber
\end{equation}

In order to determine the asymptotically decaying solutions of \eqref{eq:yinc} that have the form \eqref{yc:selfsim}, we use Laplace transform method as in \cite{Mullins1957}, \cite{Kalantarova2019}. We take the Laplace transform of \eqref{eq:yinc} with respect to the variable $\tau$
\begin{equation}\label{eq:Lap}
p\overline{y}_{c}-B\frac{\partial^{6}\overline{y}_{c}}{\partial\zeta^{6}}=0,
\end{equation}
where
\begin{equation}
\overline{y}_{c}(\zeta,p)=\int_{0}^{\infty}e^{-p\tau}y_{c}(\zeta,\tau)d\tau.\nonumber
\end{equation}

Recalling $y_{c}$ is of the form \eqref{yc:selfsim} it follows that the ODE \eqref{eq:Lap} has a set of 6 fundamental solutions given by
\begin{eqnarray}
\label{lyc1}&&\overline{y}_{c1}(\zeta,p)=B^{r}p^{-(1+r)}\exp\left(\frac{p^{1/6}}{B^{1/6}}\zeta\right),\\
\label{lyc2}&&\overline{y}_{c2}(\zeta,p)=B^{r}p^{-(1+r)}\exp\left(\frac{p^{1/6}}{2B^{1/6}}\zeta\right)\cos\left(\frac{\sqrt{3}p^{1/6}}{2B^{1/6}}\zeta\right),\\
\label{lyc3}&&\overline{y}_{c3}(\zeta,p)=B^{r}p^{-(1+r)}\exp\left(\frac{p^{1/6}}{2B^{1/6}}\zeta\right)\sin\left(\frac{\sqrt{3}p^{1/6}}{2B^{1/6}}\zeta\right),\\
\label{lyc4}&&\overline{y}_{c4}(\zeta,p)=B^{r}p^{-(1+r)}\exp\left(-\frac{p^{1/6}}{B^{1/6}}\zeta\right),\\
\label{lyc5}&&\overline{y}_{c5}(\zeta,p)=B^{r}p^{-(1+r)}\exp\left(-\frac{p^{1/6}}{2B^{1/6}}\zeta\right)\cos\left(\frac{\sqrt{3}p^{1/6}}{2B^{1/6}}\zeta\right),\\
\label{lyc6}&&\overline{y}_{c6}(\zeta,p)=B^{r}p^{-(1+r)}\exp\left(-\frac{p^{1/6}}{2B^{1/6}}\zeta\right)\sin\left(\frac{\sqrt{3}p^{1/6}}{2B^{1/6}}\zeta\right).
\end{eqnarray}
We attempt to calculate the inverse Laplace transform of each $\overline{y}_{ci}(\zeta,p)$ for $i=1,\dots,6$, by using the fact that as a self-similar solution to \eqref{eq:yinc}, which is of the form \eqref{yc:selfsim}, each $y_{ci}(\zeta,\tau)$ may be expressed as \eqref{sol:gss}. Then we have
\begin{equation}\label{d0}
\frac{\partial^{(i-1)}y_{cj}}{\partial\zeta^{(i-1)}}(0,\tau)=(B\tau)^{r+(1-i)/6}C_{i}(i-1)!,\quad i,j=1,2,\ldots,6.
\end{equation}
Taking the Laplace transform of \eqref{d0}, we get
\begin{equation}\label{ld0}
\frac{\partial^{(i-1)}\overline{y}_{cj}}{\partial\zeta^{(i-1)}}(0,p)=C_{i}(i-1)!B^{\frac{1-i}{6}+r}p^{\frac{1}{6}(-7+i-6r)}\Gamma\left(\frac{7}{6}-\frac{i}{6}+r\right).
\end{equation}
Note that the left hand side of \eqref{ld0} can be directly calculated from \eqref{lyc1}-\eqref{lyc6}, which combined with \eqref{d0} gives us

\begin{equation}
\begin{bmatrix}
y_{c1}(\zeta,\tau)\\[0.3em]
y_{c2}(\zeta,\tau)\\[0.3em]
y_{c3}(\zeta,\tau)\\[0.3em]
y_{c4}(\zeta,\tau)\\[0.3em]
y_{c5}(\zeta,\tau)\\[0.3em]
y_{c6}(\zeta,\tau)
\end{bmatrix}
=(B\tau)^{r}\begin{bmatrix}
 1   &  1                        &  1                         & 1  & 1                           &  1 \\[0.5em]
 1   & \frac{1}{2}          &-\frac{1}{2}           & -1  &-\frac{1}{2}           &\frac{1}{2}  \\[0.5em]
 0   &\frac{\sqrt{3}}{2} &  \frac{\sqrt{3}}{2}&  0  &-\frac{\sqrt{3}}{2}  &-\frac{\sqrt{3}}{2}   \\[0.5em]
 1   & -1                        &  1                         &-1  & 1                           & -1 \\[0.5em]
 1   &-\frac{1}{2}          &-\frac{1}{2}           &  1  &-\frac{1}{2}           &-\frac{1}{2}  \\[0.5em]
 0   &\frac{\sqrt{3}}{2} &-\frac{\sqrt{3}}{2} &  0  &\frac{\sqrt{3}}{2}   &-\frac{\sqrt{3}}{2}   
\end{bmatrix}
\begin{bmatrix}
 \frac{1}{0!\Gamma(1+r)}v_{1}(w)\\[0.3em]
 \frac{1}{1!\Gamma\left(\frac{5}{6}+r\right)} v_{2}(w)\\[0.7em]
 \frac{1}{2!\Gamma\left(\frac{4}{6}+r\right)}v_{3}(w)\\[0.7em]
 \frac{1}{3!\Gamma\left(\frac{1}{2}+r\right)}v_{4}(w)\\[0.7em]
 \frac{1}{4!\Gamma\left(\frac{1}{3}+r\right)}v_{5}(w)\\[0.7em]
 \frac{1}{5!\Gamma\left(\frac{1}{6}+r\right)}v_{6}(w)
\end{bmatrix}.\nonumber
\end{equation}

$y_{ci}(\zeta,\tau)$ for $i=1,2,3$ oscillate with a growing amplitude. The proof of this asymptotic behavior  (illustrated in Figure \ref{fig:yc1yc2}(A), (B), Figure \ref{fig:yc3yc4} (A)) follows directly from their definitions, and the proof of asymptotic decay of $y_{ci}(\zeta,\tau)$ for $i=4,5,6$ (illustrated in Figure \ref{fig:yc3yc4} (B), Figure \ref{fig:yc5yc6} (A), (B)) is technical and utilizes Fourier cosine transforms method and repeated application of integration by parts, \cite{Martin2009}, \cite[Appendix B]{Kalantarova2019}.

\begin{figure}[ht]
    \centering
    \begin{subfigure}[b]{0.49\textwidth}
        \includegraphics[width=\textwidth]{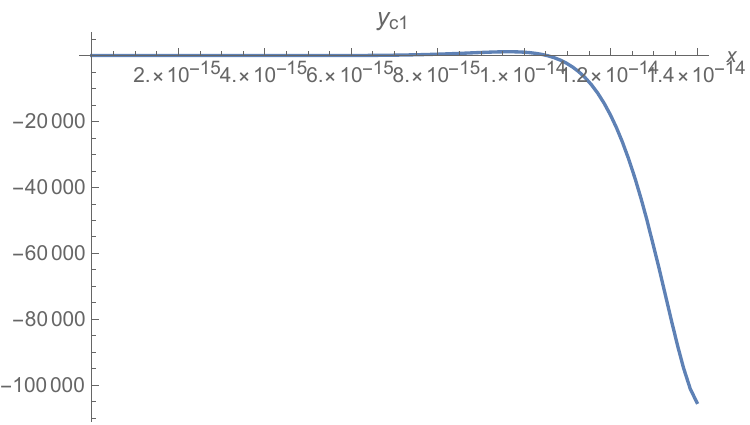}
        \caption{}
        \label{fig:yc1}
    \end{subfigure}
    \begin{subfigure}[b]{0.49\textwidth}
        \includegraphics[width=\textwidth]{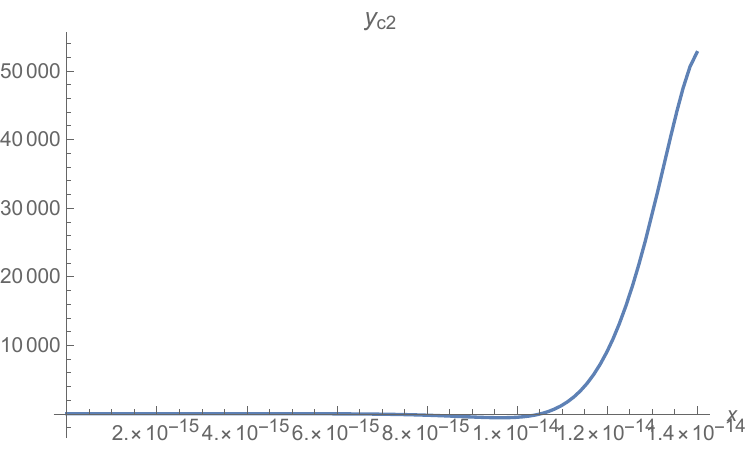}
        \caption{}
        \label{fig:yc2}
    \end{subfigure}
    \caption{\footnotesize{The plots are done for $r=-1$, $Bt=\alpha^{5}=8.587\times10^{-76}$.}}\label{fig:yc1yc2}
\end{figure}

\begin{figure}[ht]
    \centering
    \begin{subfigure}[b]{0.49\textwidth}
        \includegraphics[width=\textwidth]{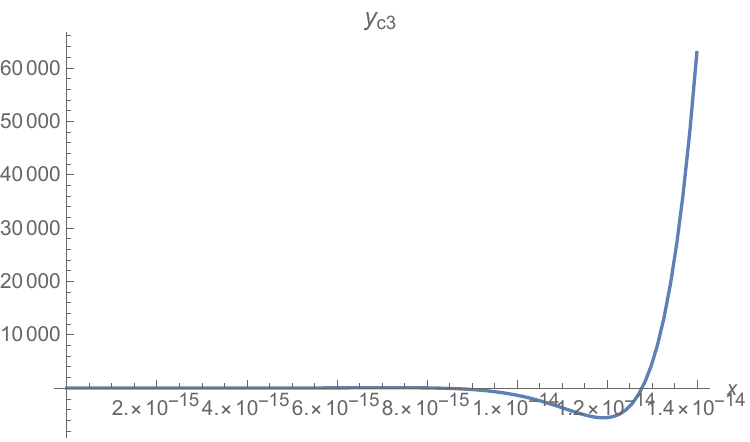}
        \caption{}
        \label{fig:yc3}
    \end{subfigure}
    \begin{subfigure}[b]{0.49\textwidth}
        \includegraphics[width=\textwidth]{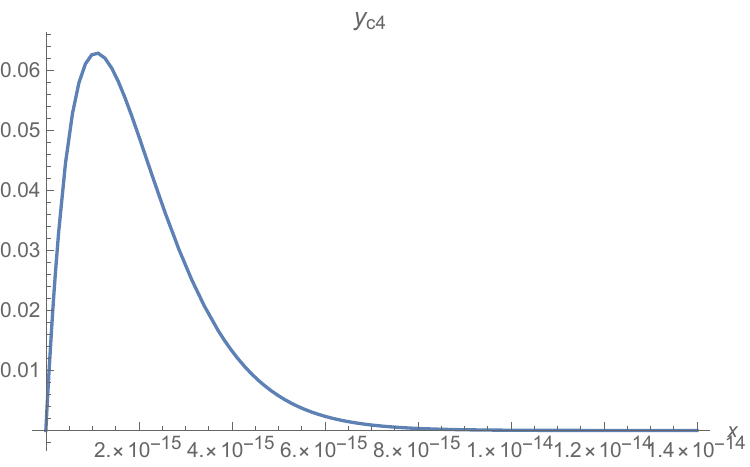}
        \caption{}
        \label{fig:yc4}
    \end{subfigure}
    \caption{\footnotesize{The plots are done for $r=-1$, $Bt=\alpha^{5}=8.587\times10^{-76}$.}}\label{fig:yc3yc4}
    \begin{subfigure}[b]{0.49\textwidth}
        \includegraphics[width=\textwidth]{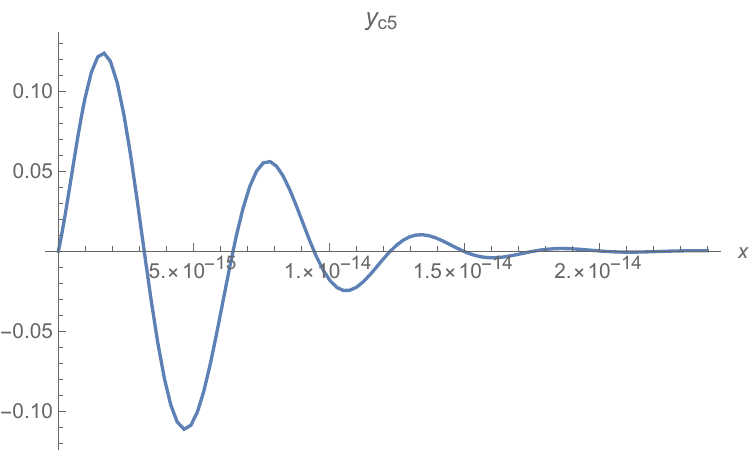}
        \caption{}
        \label{fig:yc5}
    \end{subfigure}
    \begin{subfigure}[b]{0.49\textwidth}
        \includegraphics[width=\textwidth]{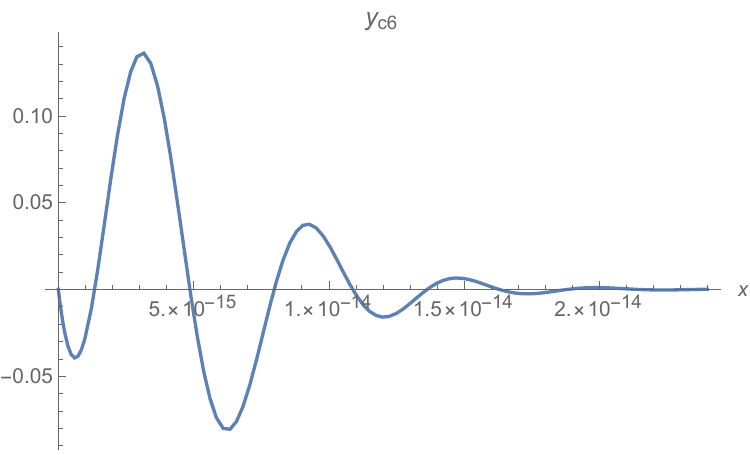}
        \caption{}
        \label{fig:yc6}
    \end{subfigure}
    \caption{\footnotesize{The plots are done for $r=-1$, $Bt=\alpha^{5}=8.587\times10^{-76}$.}}\label{fig:yc5yc6}
\end{figure}

Thus a solution of \eqref{eq:yinc} of the form \eqref{yc:selfsim} that satisfy \eqref{hyp:bc1}-\eqref{hyp:bc2} can be written as
\begin{equation}
y_{c}(\zeta,\tau)=\sum_{i=1}^{6}c_{i}y_{i}(\zeta,\tau),
\end{equation}
where $c_{1}=c_{2}=c_{3}=0$ and $c_{4}$, $c_{5}$, $c_{6}$ satisfy
\begin{eqnarray}
&&-c_{4}-\frac{1}{2}c_{5}+\frac{\sqrt{3}}{2}c_{6}=V'(0)\Gamma\left(\frac{5}{6}+r\right),\nonumber\\
&&-c_{4}+c_{5}=\alpha(B\tau)^{1/3}V'(0)\Gamma\left(\frac{1}{2}+r\right),\nonumber\\
&&-c_{4}-\frac{1}{2}c_{5}-\frac{\sqrt{3}}{2}c_{6}=\alpha^{2}(B\tau)^{2/3}V'(0)\Gamma\left(\frac{1}{6}+r\right).\nonumber
\end{eqnarray}
Solving the above algebraic system for $c_{4}$, $c_{5}$, $c_{6}$ gives us the desired result

\begin{multline}\label{sol:yc}
y_{c}(\zeta,\tau)=-\frac{V'(0)}{\sqrt{3}}\left[\alpha^{2}(B\tau)^{\frac{2}{3}}\Gamma\left(\!\frac{1}{6}+r\!\right)-\Gamma\left(\!\frac{5}{6}+r\!\right)\right]y_{c6}(\zeta,\tau)\\
-\frac{V'(0)}{3}\!\left[\alpha^{2}(B\tau)^{\frac{2}{3}}\Gamma\left(\!\frac{1}{6}+r\!\right)-2\alpha(B\tau)^{\frac{1}{3}}\Gamma\left(\!\frac{1}{2}+r\!\right)+\Gamma\left(\!\frac{5}{6}+r\!\right)\right]y_{c5}(\zeta,\tau)\\
-\frac{V'(0)}{3}\!\left[\alpha^{2}(B\tau)^{\frac{2}{3}}\Gamma\left(\!\frac{1}{6}+r\!\right)+\alpha(B\tau)^{\frac{1}{3}}\Gamma\left(\!\frac{1}{2}+r\!\right)+\Gamma\left(\!\frac{5}{6}+r\!\right)\right]y_{c4}(\zeta,\tau).
\end{multline}

Note that in case of $r=-\frac{1}{6}$, integrating \eqref{eq:hyp} term by term implies 
\begin{equation}
V'(0)=V'''(0)=V^{(5)}(0)=0,
\end{equation}
which results in the trivial solution
\begin{equation}
y_{c}(\zeta,\tau)\equiv0.\nonumber
\end{equation}
Moreover, $r<-\frac{2}{3}$ since $y_{c}$ must satisfy \eqref{cl:dect}, i.e., it must be exponentially small outside the corner layer.

Furthermore, it follows from Remark \ref{cl:small} that $V'(0)=O(\alpha)$.

\end{proof}

\section*{Acknowledgments}
The first author would like to thank Dr.~Orestis Vantzos for bringing the method of singular perturbation to their attention and for helpful discussions.

This version of the article has been accepted for publication, after peer review and is subject to Springer Nature’s AM terms of use, but is not the Version of Record and does not reflect post-acceptance improvements, or any corrections. The Version of Record is available online at: https://doi.org/10.1007/s00161-021-01040-0


\clearpage
    
\bibliography{bibl_hvk}
\bibliographystyle{amsplain}

\end{document}